\pdfoutput=1
\documentclass[microtype]{gtpart}
\usepackage{graphicx,pinlabel}
\usepackage[all]{xy}
\usepackage{amscd}
\usepackage[mathscr]{eucal}
\usepackage[top=2.7cm, left=4cm, right=4cm, bottom=5cm]{geometry}
\usepackage{subfig}
\usepackage{color}
\usepackage{enumitem}

\newcommand{\bz}{\mathbb Z}

\newcommand{\cN}{\mathcal N}

\newcommand{\cc}{\mathcal C}
\newcommand{\cd}{\mathcal D}

\newcommand{\al}{\alpha}
\newcommand{\be}{\beta}

\DeclareMathOperator{\Aut}{Aut}

\DeclareMathOperator{\Mod}{Mod}

\DeclareMathOperator{\Link}{Link}

\newcommand{\cb}{\mathcal{CB}}

\newcommand{\tc}{\mathcal{TC}}

\newtheorem{Thm}{Theorem}[section]
\newtheorem{Prop}[Thm]{Proposition}
\newtheorem{Lem}[Thm]{Lemma}
\newtheorem{Cor}[Thm]{Corollary}

\newtheorem{Fact}[Thm]{Fact}

\theoremstyle{definition}
\newtheorem{Def}[Thm]{Definition}

\theoremstyle{remark}

\numberwithin{equation}{section}

\newtheorem*{namedtheorem}{\theoremname}
\newcommand{\theoremname}{testing}

\makeatletter
\newtheorem*{rep@theorem}{\rep@title}
\newcommand{\newreptheorem}[2]{%
\newenvironment{rep#1}[1]{%
 \def\rep@title{#2 \ref{##1}}%
 \begin{rep@theorem}}%
 {\end{rep@theorem}}}
\makeatother

\newreptheorem{theorem}{Theorem}
\newreptheorem{lemma}{Lemma}
\newreptheorem{proposition}{Proposition}

\title{Automorphisms of the compression body graph}

\author{Ian Biringer}
\givenname{Ian}
\surname{Biringer}
\address{Department of Mathematics \\ Boston College}
\email{ian.biringer@bc.edu}
\urladdr{https://www2.bc.edu/ian-p-biringer/}

\author{Nicholas G.~Vlamis}
\givenname{Nicholas}
\surname{Vlamis}
\address{Department of Mathematics \\ University of Michigan}
\email{vlamis@umich.edu}
\urladdr{http://nickvlamis.com}

\begin{document}

\begin{abstract}
When $S$ is a closed, orientable surface with genus $g(S) \geq 2$, we show that the automorphism group of the compression body graph $\cb(S)$ is the mapping class group.  Here, vertices are compression bodies with exterior boundary $S$, and edges connect pairs of compression bodies where one contains the other.
\end{abstract}

\maketitle

\section{Introduction}

A \textit{compression body} is a compact, orientable, irreducible $3$-manifold $C$ with a distinguished `exterior' boundary component $\partial_+ C$, such that the inclusion $\partial_+ C \longrightarrow C$ is a $\pi_1$-surjective.  Fixing a closed, orientable surface $S $, an \textit{$S$-compression body} is a pair $(C,f)$ where $C$ is a compression body and $f\co S \to \partial_+C$ is a homeomorphism. 

Any $S$-compression body can be constructed as follows, see Lemma \ref{compressdiscs}. Starting with $S\times [0,1]$, attach $2$-handles along a collection of disjoint essential annuli in $S\times\{0\}$ and then glue a 3-ball onto every resulting spherical boundary component.  Here, the exterior boundary is $S\times\{1\}$, which clearly $\pi_1$-surjects, and has a natural identification with $S $.  Two extreme examples of this construction occur when the collection of annuli is empty, in which case we obtain the \emph {trivial compression body} $S \times [0,1]$, and when the collection is large enough so that after attaching the two-handles, \emph{every} interior boundary component is a sphere, in which case  $\partial C=S \times \{ 1\}$ and $C$ is a \emph {handlebody}.

Two $S$-compression bodies $(C,f)$ and $(D,g)$ are \emph {isomorphic} if there is a homeomorphism $H : C \longrightarrow D$ such that $H \circ f =g $. We also say that $(C,f)$ \emph {is contained in} $(D,g)$ if there is an embedding $H : C \longrightarrow D$ such that $H \circ f =g $. It follows (see \S \ref{csec}) that $(C,f)$ and $(D,g)$ are isomorphic if and only if each is contained in the other.

The \emph {compression body graph}, written ${\mathcal {CB}}(S)$, is the graph whose vertices are isomorphism classes of nontrivial $S$-compression bodies, and where $(C,f),(D,g)$ are adjacent if either 
$$(C,f) \subset (D,g) \ \text{ or } \ (D,g) \subset (C,f).$$

The \emph {mapping class group} of $S $, written $\Mod(S)$, is the group of isotopy classes of self-homeomorphisms $\phi$ of $S$. It acts on $\mathcal {CB}(S)$ by precomposing the markings: $$(C,f) \overset{\phi}{\longrightarrow} (C,f\circ \phi^{-1}).$$

\begin {Thm}\label{main}
When $g(S)\geq 2$, the natural map $\Mod(S) \longrightarrow \Aut(\mathcal {CB}(S))$ is a surjection.
\end {Thm}

Here $g(S)$ is the genus of the surface $S$. Note that when $S$ is a torus the theorem is false, since then ${\mathcal {CB}}(S)$ is an infinite graph with no edges. 

The action of $\Mod(S)$ is faithful except when $S$ has genus two, in which case the kernel is generated by the hyperelliptic involution. This follows from the analogous statement about the action of the mapping class group on the complex of curves, since any simple closed curve $\alpha$ on $S$ gives a \emph {small compression body} $S[a]$ obtained by attaching a two-handle along an annulus framing $\alpha$ on $S\times \{0\}$, and $\Mod(S)$ acts on these small compression bodies via the defining curves, see \S \ref{csec}. 

Metrically, the compression body graph is $\delta$-hyperbolic and has infinite diameter.  This follows from as yet unpublished work of Maher-Schleimer, who study a \emph {handlebody graph} that is quasi-isometric to ${\mathcal {CB}}(S)$.  We find the fine structure of ${\mathcal {CB}}(S)$ more natural, but Maher-Schleimer should be credited as the first to study the notion of distance between handlebodies or compression bodies defined by such graphs.

The compression body graph is an example of a \emph {comparability graph}, where an edge joins vertices that are comparable in a partial order. As such, it is \emph {perfect,} i.e.\ the chromatic and clique numbers of all subgraphs agree. Such graph invariants come up briefly below; for instance, Lemma \ref{lem:chromatic} implies that the chromatic and clique numbers of $\cb(S)$ are $2g-1$.

The inspiration for Theorem \ref{main} is the celebrated theorem of Ivanov \cite{Ivanovautomorphisms}, see also Luo \cite{Luoautomorphisms}, that the automorphism group of the curve graph is $\Mod(S) $. Here, the \emph{curve graph} is the graph $\cc(S)$ whose vertices are isotopy classes of simple closed curves on $S$, and edges connect isotopy classes that admit disjoint representatives. Ivanov used his theorem to conclude that the isometry group of Teichm\"uller space, regarded with the Teichm\"uller metric, is also $\Mod(S) $, and that the outer automorphism group of the mapping class group is trivial. Since then, there have been a number of papers proving similar rigidity results for complexes associated to a surface $S$, e.g.\ the complex of non-separating curves \cite{Irmakcomplexes}, and the pants complex \cite{Margalitautomorphisms}.  

The action of the mapping class group on ${\mathcal {CB}}(S)$ encodes a wealth of information about the interaction of mapping classes and $3$-manifolds. For instance, an element $\phi \in \Mod(S)$ fixes an $S$-compression body $(C,f)$ if and only if the homeomorphism $f \circ \phi \circ f^{-1}$ of $\partial_+ C $ extends to a homeomorphism of $C$. Extension into compression bodies has been previously studied by Casson-Long \cite{Cassonalgorithmic}, Long \cite{Longdiscs,Longbounding}, Biringer-Johnson-Minsky \cite{Biringerextending} and Ackermann \cite{Ackermannalternative}, among others. In studying the cobordism group of surface automorphisms, Bonahon \cite[Prop 5.1]{Bonahoncobordism} shows that when a homeomorphism of a surface $S $ extends to a $3$-manifold $M$ with $\partial M=S$, it also extends to a $3$-manifold in which all the non-periodic action happens on the union of a compression body and an interval bundle. 

For the proof of Theorem \ref{main}, we introduce an auxiliary simplicial complex, which is of independent interest. The \textit{torus complex}, denoted $\tc(S)$, is the simplicial complex whose vertices are isotopy classes of non-separating simple closed curves, and where a collection of vertices $\{a_0, \ldots, a_k\}$ spans a $k$-simplex if there exists a punctured torus $T\subset S$ such that $a_i$ can be isotoped to be contained in  $T$ for all $1\leq i \leq k$.  

\begin{Thm}
\label{thm:torus-aut}
For $g(S)\geq 2$, the natural map $\Mod(S) \to \Aut(\tc(S))$ is a surjection.
\end{Thm}

In other words, every bijection of the set of non-separating simple closed curves on $S$ that preserves when curves lie in a punctured torus is given by a mapping class.

As in Theorem \ref{main}, this map is an isomorphism except when $S$ has genus two, in which case the kernel is generated by the hyperelliptic involution.  The relationship between $\tc(S)$ and $\cb(S)$ is described in the following proof sketch.

\subsection{Sketch of the proof of Theorem \ref {main}}

We will outline here the proof of the main theorem, modulo results to be proved later. A full proof will be given at the end of the paper in Section \ref{sec:proof}.

Suppose that $f : \cb(S)\longrightarrow\cb(S)$ is an automorphism. In Proposition \ref{prop:small-invariant}, we show that $f$ preserves the set of small compression bodies $S[a]$, those that are obtained from $S$ by compressing a single curve $a$. Moreover, $f$ preserves whether the compressing curve is non-separating or separating. Briefly, the idea is that small compression bodies are (among) those with small \emph{height}, a notion of complexity introduced in \S\ref{sec:height}, and that the height of a compression body $C$ is encoded in the chromatic number of certain subsets of the link of $C\in \cb(S)$. This is the subject of Section \ref{small}.

In particular, $f$ acts on the set of non-separating simple closed curves on $S$. This action has the property that it preserves when a set of non-separating curves comes from a single punctured torus $T\subset S$.  When $g(S)\geq 3$, this is because two non-separating curves $a,b$ lie in a punctured torus if and only if the compression bodies $S[a]$ and $S[b]$ contain a common sub-compression body, while $g(S)=2$ requires an additional argument. This leads us to consider the \emph{torus complex} $\tc(S)$.

Section \ref{torus} is dedicated to proving Theorem \ref{thm:torus-aut} and is entirely separate from the rest of the paper.
Consequently, the action of $f$ on the set of non-separating small compression bodies agrees with the action of mapping class $\phi\in \Mod(S)$. We then show that the actions of $f$ and $\phi$ agree on all of $\cb(S)$, using that a compression body is determined by the small compression bodies it contains.

\subsection{Acknowledgements} The authors would like to thank Joseph Maher and Saul Schleimer for helpful conversations. The first author was partially supported by NSF grant DMS-1308678.

\section{Compression bodies}
\label{csec}

A \textit{compression body} is a compact, orientable, irreducible $3$-manifold $C$ with a $\pi_1$-surjective boundary component $\partial_+ C$, called the \emph {exterior boundary} of $C$. The complement $\partial C \smallsetminus \partial_+ C$ is called the \emph {interior boundary}, and is written $\partial_- C$. Note that the interior boundary is incompressible.  For if an essential simple closed curve on $\partial_- C$ bounds a disk $D \subset C$, then $C \smallsetminus  D$ has either one or two components, and in both cases, Van Kampen's Theorem implies that $\partial_+ C$, which is disjoint from $D$, cannot $\pi_1$-surject. 

Let $S$ be a closed, orientable surface.  In the introduction, we defined an \textit{$S$-compression body} as a pair $(C,f)$ where $f\co S \longrightarrow \partial_+C$ is a homeomorphism. 
Throughout the rest of the paper, we will suppress the marking $f$, and consider compression bodies whose exterior boundaries are implicitly identified with $S$.  
With this new language, two $S$-compression bodies are \emph{isomorphic} if they are homeomorphic via a map that is the identity on their exterior boundaries, and an $S$-compression body $C$ is \emph {contained in} $D$, written $C \subset D$, if there is an embedding of $C$ into $D$ that is the identity on the exterior boundary. Often, we will just view $C $ as a submanifold of $D$ that shares its exterior boundary.

 If $\{a_1, \ldots, a_k\}$ is a collection of disjoint simple closed curves on $S$, let $S[a_1, \ldots, a_k]$ be the $S$-compression body obtained by \emph {compressing} each of the curves $a_i$. This means that we attach two-handles to $S \times [0,1]$ along a collection of annuli on $S\times \{0\}$ whose core curves are the $a_i $, fill in $S^2$-boundary components with balls, and identify $S$ with $S \times\{1\}$. We will call $\{a_1, \ldots, a_k\}$ a \textit{compressing system} for $S[a_1, \ldots, a_k]$.

A simple closed curve on $S$ is called a \emph {disk}, or \emph {meridian}, of an $S$-compression body $C$ if it bounds an embedded disk in $C$. The \emph {disk set} of an $S$-compression body $C$, written $\cd(C)$, is the set of (isotopy classes of) meridians of $C$. 

\begin {Lem}\label {compressdiscs}
If $\{a_1,\ldots, a_k\} \subset \cd(C)$ is a collection of disjoint meridians of $C$, then $S[a_1, \ldots, a_k] \subset C$. Moreover, if $\{a_1,\ldots, a_k\} $ is maximal, $S[a_1, \ldots, a_k] = C$.
\end {Lem}

In particular, any compression body can be constructed from $S$ as above, by compressing a collection of simple closed curves and filling in spheres with balls.
\begin {proof}
As the meridians $\{a_1,\ldots, a_k\} $ are disjoint, they bound a collection of disjoint disks $D_i$ in $C$. By irreducibility, every $2$-sphere boundary component of a neighborhood of the union $S \cup \bigcup_i D_i$ bounds a ball in $C $. So, filling in these boundary components gives a submanifold of $C$ homeomorphic to $S[a_1, \ldots, a_k]$.

Now assume that the collection $\{a_1,\ldots, a_k\} $ is maximal. The interior boundary components of $S[a_1, \ldots, a_k]$ are then incompressible in $C$: if not, a simple closed curve on an interior boundary component that compresses in $C \setminus S[a_1, \ldots, a_k]$ can be homotoped to a new meridian on $S$ that is disjoint from the collection $\{a_1,\ldots, a_k\} $. So, each component of $C \setminus S[a_1, \ldots, a_k]$ has a $\pi_1$-surjective, incompressible boundary component, so is an interval bundle by Waldhausen's Cobordism Theorem \cite{Waldhausenirreducible}.
\end {proof}

The following is an immediate consequence of Lemma \ref{compressdiscs}.

\begin {Cor}\label {disccont}
Let $C,D$ be $S$-compression bodies. Then $C$ and $D$ are isomorphic if and only if $\cd(C)=\cd(D)$, and $C\subseteq D$ if and only if $\cd(C)\subseteq \cd(D)$.
\end {Cor}

In particular, this implies that $C$ and $D$ are isomorphic if and only if $C\subseteq D $ and $D \subseteq C$.

Compression bodies can also be constructed as boundary connected sum of closed balls and interval bundles $F_i \times [0,1]$, where the $F_i$ are closed, orientable surfaces, and the boundary connected sums are always performed along $F_i \times \{1\}$. See Figure \ref {alien}. Here, such sums of irreducible $3$-manifolds are irreducible, and the union of the $F_i \times \{1\}$ with the boundaries of the balls and the $1$-handles is a $\pi_1$-surjective boundary component.

Lemma \ref{compressdiscs} shows that every compression body can be so constructed. If $C=S[a_1, \ldots, a_k] $, let $F_i$ be the surfaces obtained by \emph{surgering} $S$ along disks  $D_1,\ldots,D_k \subset C$ with boundary $a_1,\ldots,a_k$. (Here, if $S_i$ is a component of $S\setminus \left (a_1 \cup \cdots \cup a_k\right)$, then $F_i$ is obtained from $S_i$ by attaching the adjacent disks to its boundary components.) Each $F_i\subset C$ bounds either a ball, if $F_i$ is a sphere, or an interval bundle $F_i \times [0,1]$. These pieces are attached along the disks $D_j$, which expresses $C $ as a boundary connected sum.

\begin {figure}
\centering
\includegraphics{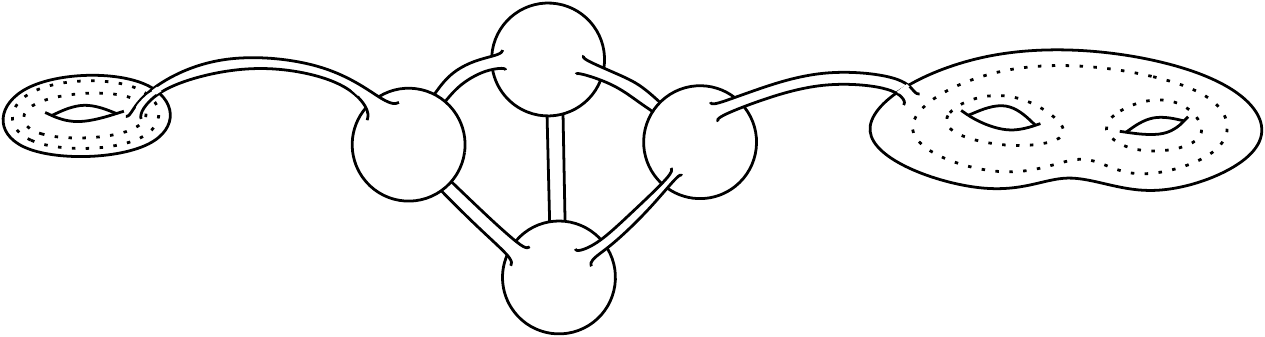}
\caption {A boundary connected sum of four balls and interval bundles over a torus and a genus two surface. Here, the boundary connected sum of the four balls is a genus 2 handlebody.}
\label {alien}
\end {figure}
\begin {Cor}\label {onlyone}
A compression body is determined up to homeomorphism by the genera of its boundary components.
\end{Cor}

Note that since the exterior boundary always has the largest genus, it is not necessary to label the genera as `exterior' and `interior' in the statement of the corollary.

\begin {proof}
We claim that any compression body $C$ can be obtained by attaching interval bundles $F_i \times [0,1]$ with a \emph{single} $1$-handle to a handlebody $H$. To do this, think of a boundary connected sum decomposition for $C$ as a graph, where vertices are balls and interval bundles, and edges are $1$-handles. The homeomorphism type of the compression body is unchanged if each interval bundle vertex in the graph is replaced by a ball vertex, and then that interval bundle is reattched to the new ball with an additional $1$-handle. The result is a graph of balls and $1$-handles, i.e.\ a handlebody, attached to interval bundles.

The genera of the interior boundary components determine the interval bundles, while the difference between the genus of the exterior boundary and the total genus of the interior boundary is the genus of the handlebody. 
\end {proof}

Finally, we end with a useful gluing construction.

\begin {Cor}[Exterior-to-interior gluings]\label {sum} Suppose that $C$ is a $S$-compression body with an interior boundary component $F \subset\partial_- C$, and that $D$ is an $F $-compression body. Then the natural gluing $C \sqcup_F D$ is an $S$-compression body.

Conversely, let $C \subset E$ be $S$-compression bodies, and let $\partial_- C = F_1 \sqcup \ldots \sqcup F_n$. Then $E$ is isomorphic to an $S$-compression body obtained by gluing to $C$ a collection of (possibly trivial) $F_i$-compression bodies $D_i$, one for each $i$.
\end {Cor}
\begin {proof}
Represent $C$ and $D$ as boundary connected sums of balls and interval bundles. Gluing $F \times [0,1]$ to $D$ does not change its homeomorphism type, so $C \sqcup_F D$ is a boundary connected sum of the balls and interval bundles from $D$, together with all balls and interval bundles from $C$ except $F\times [0,1]$.

For the second part, extend a compressing system $a_1,\ldots,a_k$ for $C$ to a compressing system $a_1,\ldots,a_k,b_1,\ldots,b_l$ for $D $. The $b_j$ are all disjoint from $a_1,\ldots,a_k$, so are homotopic to simple closed curves $b_j'$ on the interior boundary of $C $. Then $C_i $ is the compression body defined by the compressing system consisting of all $b_j' $ that lie on $F_i $.
\end {proof}

\subsection{Small compression bodies}

Throughout this work, there will be a special class of compression bodies that we will consistently come back to, which we now define:

\begin{Def}
A \emph {small compression body} is a compression body $C$ that can be written as $S[a]$ for some simple closed curve $a\subset S=\partial_+ C$. 
\end{Def}

A solid torus is an example of a small compression body -- it has a unique meridian. 

When $S$ has genus at least two, the disk set of a small compression body $S[a]$ has a unique meridian only when $a$ is separating. We will prove this, but first we need some notation. If $a,b \in \cc(S)$ and $i(a,b) = 1$, the \textit{band sum of $a$ and $b$} is the separating curve $$B(a,b) = \partial N(a\cup b),$$ where $N(a \cup b)$ is a regular neighborhood of $a\cup b$.  Note that $B(a,b)$ is the boundary of a once-punctured torus, $N(a \cup b)$, that contains $a$. Conversely, any curve that bounds a once punctured torus $T$ containing $a $ can be expressed as a band sum $B(a,b)$, by taking $b$ to be any curve in $T$ that intersects $a $ once.

\begin{Prop}[Disk sets of small compression bodies]\label {smalldisc} Suppose that $S$ is a closed, orientable surface and $a$ is a simple closed curve on $S $. If $S$ is a torus or $a$ is separating, $$\cd(S[a]) = \{a\},$$ while if the genus $g(S)\geq 2$ and $a$ is non-separating, then \begin{align*}
\cd(S[a]) &= \{a\} \cup \{ B(a,b) :  b\in \cc(S),\ i(a,b)=1\} \\
&= \{a\} \cup \{ \partial T \colon T \subset S \text{ a punctured torus with } a \subset T \}.
\end{align*}
\end{Prop}

In the rest of the paper, we will call a small compression body $C=S[a]$ \emph {separating} or \emph{non-separating} depending on the type of the compressed curve $a \subset S$.

An $S$-compression body $C $ is called \emph {minimal} if it does not contain any nontrivial sub-compression bodies. Any minimal compression body must be small, but if $a$ is non-separating and $g(S)\geq 2$ then $S[a]$ is not minimal, since compressing any separating meridian gives a nontrivial sub-compression body. On the other hand, all other small compression bodies have a single meridian, so are certainly minimal. In summary:

\begin {Cor}\label {cor:minimal}
An $S$-compression body is minimal if and only if it is a solid torus or a small compression body obtained by compressing a separating curve.
\end{Cor}

Before proving Proposition \ref{smalldisc}, we need the following lemma. Although we are only concerned with compression bodies, we might as well state it more generally.

\begin {Lem}\label {innermost}
Suppose that $S$ is a boundary component of a compact $3$-manifold $C$, and $a,b$ are meridians on $S$ with $i(a,b)>0$. Then the intersections with $a$ divide $b$ into a collection of arcs, one of which, say $b'$, has the following properties:
\begin {enumerate}
\item both intersections of $b'$ with $a$ happen on the same side of $a$,
\item the union of $b'$ with either of the two arcs of $a$ with the same endpoints is a meridian, which is  disjoint from (after isotopy) but not isotopic to $a$.
\end {enumerate}
\end {Lem}
\begin{proof}
Pick two transverse disks $D_a$ and $D_b$ with boundaries $a$ and $b$, and assume that number of components of the intersection $D_a \cap D_b$ is minimal.  

Let $\gamma $ be an arc of $D_a \cap D_b$ that is \emph {innermost} in $D_b$, meaning that one of the two components, say $X$, of $D_b \smallsetminus \gamma $ has no intersections with $D_a$. The boundary arc $b'=X\cap \partial D_b$ is disjoint from $a$ except at its endpoints. These two intersections happen on the same side of $a$, since the side of the disk $D_a$ that $X$ is on cannot flip while traversing $\gamma$. 

Gluing $X$ to either of the two components of $D_a \smallsetminus \gamma$ gives a disk $D \subset C$ whose boundary is the union of $b'$ with an arc $a'$ of $a$, as desired.  Note that $\partial D$ is an essential simple closed curve in $S$, since if it were inessential $b'$ and $a'$ would be homotopic rel endpoints, and then $a$ and $b$ would not be in minimal position. Also, since the intersections of $b'$ with $a$ happen on the same side of $a$, $\partial D$ can be isotoped to be disjoint from $a$.

Hoping for a contradiction, assume that $\partial D$ is isotopic to $a$. By a small isotopy, $\partial D$ can be made disjoint from $a$; more carefully, perform the isotopy by pushing $a'$ slightly away from $a$ while keeping its endpoints on $b'$. After the isotopy, $\partial D$ and $a$ cobound an annulus $A\subset S$. If $A$ and $b'$ approach $a$ from the same side, then $b'$ is homotopic to $a \smallsetminus a'$ rel endpoints (see Case 1, Figure \ref{fig1}), so as before $a,b$ cannot be in minimal position.

\begin {figure}
\centering
\includegraphics{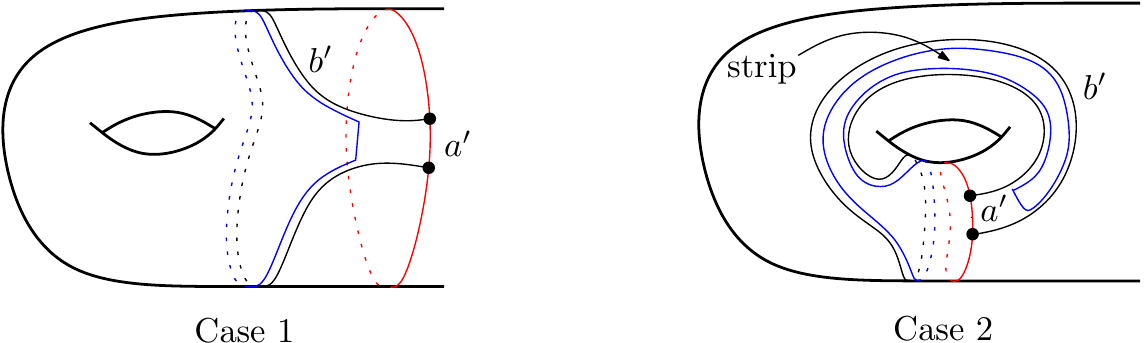}
\caption {The curve $a$ is drawn in red, and after an isotopy to make it disjoint from $a$, the curve $\partial D$ is drawn in blue. In both cases, the annulus $A$ bounded by $a$ and $\partial D$ starts from the left of $a$.}
\label {fig1}
\end {figure}

If $A$ and $b'$ approach $a$ from opposite sides, then $a$ is nonseparating and up to the action of the mapping class group of $S$, the picture is exactly as in Case 2, Figure \ref{fig1}. However, it is impossible to extend the $b'$ in this picture to a closed curve that is in minimal position with respect to $a$. For continuing from the endpoints of $b'$, since it cannot turn immediately back to intersect $a$ again, the curve $b$ would be forced to wind infinitely many times through the thin `strip' indicated in the picture, and could never close up.
\end{proof}

We are now ready to characterize the disk sets of small compression bodies.

\begin {proof}[Proof of Proposition \ref {smalldisc}]
Suppose first that $a\subset S $ is a separating curve. If $D$ is disk in $S[a]$ with boundary $a $, and  $S_i$ is a component of $S\smallsetminus a$, then the union $D\cup S_i$ is a closed surface isotopic in $C$ to an interior boundary component of $S[a]$. As remarked at the beginning of \S \ref{csec}, this means that $D\cup S_i$ is incompressible in $S[a]$. So, the only essential simple closed curves on $S_i$ that are compressible in $S[a]$ are isotopic to the boundary, $a$.  In other words, there are no other meridians of $S[a]$ that are disjoint from $a$.  A priori, there could be meridians that intersect $a$, but Lemma \ref{innermost} converts these to meridians disjoint from $a$, so in fact $a$ is the only meridian.

Now suppose that $a \subset S $ is non-separating. Form a closed surface $S'$ by attaching to $ S \smallsetminus a$ two copies of the disk $D$. As before, $S'$ is incompressible in $S[a]$. So, if $\gamma $ is a meridian in $ S \smallsetminus a$, then $\gamma $ bounds a disk in $S'$. The intersection of this disk with $ S \smallsetminus a$ is a twice punctured disk with $\gamma$ as a boundary component, and re-identifying the two copies of $a$ gives a punctured torus $T\subset S$ bounded by $\gamma$ that contains $a$.
\end {proof}

\subsection{Height of a compression body}
\label {sec:height}

When $C $ is an $S$-compression body, a \emph{sequence of minimal compressions} for $C$ is a chain
\begin {equation}\label {mincomp}S \times [0,1]=C_0 \subset C_1 \subset \cdots \subset C_k=C\end {equation}
of compression bodies in which each $C_{i}$ is created from $C_{i-1}$ by gluing a minimal $F_i$-compression body to some component $F_i \subset\partial_- C_{i-1}$. Recall that a compression body is \emph {minimal} if it does not contain any nontrivial sub-compression bodies --- in Corollary~\ref{cor:minimal} we saw that these are exactly the solid tori and separating small compression bodies.

Sequences of minimal compressions are exactly chains $(C_i)$ as in \eqref{mincomp} that are maximal, in the sense that they are not properly contained in a larger chain. 

As an example, let $C=S[a]$, where $a$ is nonseparating and $g(S)\geq 2$.  By Proposition~\ref{smalldisc}, any separating meridian $b$ for $S[a]$ bounds a punctured torus containing $a$, so $a$ is isotopic to a curve $a'$ on a torus $T\subset \partial_- S[b]$. The compression body $S[a]$ is obtained by attaching a solid torus to $S[b]$ along $T$ so that the meridian is identified with $a'$. So here,
$$S \times [0,1] \subset S[b] \subset S[a]$$
is a sequence of minimal compressions for any separating meridian $b$ in $S[a]$.

More generally, we have the following:

\begin {Lem}\label {lem:curvetosmc}
Suppose that $C=S[a_1,\ldots,a_k]$, and for each $i$ let $S_i$ be the component of $S \setminus a_1 \cup \cdots \cup a_{i-1}$ containing $a_i$. If 
\begin {enumerate}
\item[(*)] for each $i$, either we have $g(S_i)=1$, or we have $g(S_i)\geq 2$ and $a_i$ separates $S_i$,
\end {enumerate}
then $C_i=S[a_1,\ldots,a_i]$ defines a sequence of minimal compressions for $C $. Conversely, any sequence of minimal compressions for $C$ can be written as $C_i=S[a_1,\ldots,a_i]$ for some collection $a_1,\ldots,a_k$ satisfying $(*)$.
\end{Lem}

Whenever $C=S[a_1,\ldots,a_k]$, the $(a_i)$ can be altered to satisfy $(*)$. For if $g(S_i)=0$, then $a_i$ is already a meridian in $S[a_1,\ldots,a_{i-1}]$, so its inclusion is redundant and it can be removed. If $g(S_i)\geq 2$ and $a_i$ is non-separating in $S_i$, insert a new curve $b_i\subset S_i$ that bounds a punctured torus containing $a_i$ between $a_{i-1},a_i$ in the sequence.

\begin {proof}
Suppose $C=S[a_1,\ldots,a_k]$ and $(a_i)$ satisfies $(*)$. Then for each $i $,
$$C_{i}=C_{i-1} \sqcup_{F_i} F_i[a_i'],$$ where $F_i$ is the component of $\partial_- C_{i-1}$ homotopic to the surface obtained by attaching disks to $S_i$, and $a_i' \subset F_i$ is the unique curve homotopic to $a_i$. Note that $g(S_i)=g(F_i)$, and $a_i$ separates $S_i$ if and only if $a_i'$ separates $F_i$. Then $(*)$ says that $F_i[a_i']$ is minimal.

Conversely, if $(C_i)$ is a sequence of minimal compressions, we can use Lemma \ref{compressdiscs} to iteratively extend compressing systems from $C_{i-1}$ to $C_i$. The result is a compressing system $a_1,\ldots,a_k$ for $C $ that satisfies $(*)$.
\end {proof}

The following is the main result of the section.

\begin{Prop}\label{prop:height}
If $C$ is a compression body with $\partial_-C = F_1 \sqcup\cdots \sqcup F_n$, the length $k$ of any sequence of minimal compression $S \times [0,1] =C_0 \subset C_1\subset \cdots \subset C_k=C$ is 
\begin{equation}
\label{eq:height}
\mathfrak{h}(C) := \left(2\cdot g\left(S\right) -1\right) - \sum_{i=1}^n\left(2\cdot g\left(F_i\right)-1\right).
\end{equation}
\end{Prop}

We call $\mathfrak{h}(C) $ the \emph {height} of $C$.  A genus $g$ handlebody has height $2g-1$, so a solid torus has height $1$. A separating small compression body also has height $1$, as the genera of the two interior boundary components sum to the genus of the exterior boundary.

\begin {proof}
If $C$ is a compression body and $F$ is a component of $\partial_- C$, then $\mathfrak{h}$ adds when an $F$-compression body $D$ is glued to $C$:
$$  \mathfrak{h}(C \sqcup_F D)=  \mathfrak{h}(C)+  \mathfrak{h}(D),$$
since the only boundary component from $C$ or $D$ that is not referenced in $  \mathfrak{h}(C \sqcup_F D)$ is $F$, but $2 \cdot g\left(F\right)-1$ appears with opposite signs in $  \mathfrak{h}(C)$ and $  \mathfrak{h}(D)$. So, gluing on a solid torus or a separating small compression body increments height.
\end {proof}

As a consequence of Proposition \ref{prop:height}, height is positive and increases under inclusion: $$C\, \subset \, D\ \  \Longrightarrow \ \ \mathfrak{h}(C)<\mathfrak{h}(D).$$ Hence, the length of any chain $(C_i)$ of sub-compression bodies of $C$ is at most $\mathfrak{h}(C) $, and any chain is contained in a maximal chain, i.e.\ a sequence of minimal compressions.

\begin {Cor}[`Short' compression bodies]
\label{cor:height}

\begin {enumerate}
\item Height one compression bodies are solid tori and small compression bodies $S[a]$, where $a$ is separating.
\item Height two compression bodies are small compression bodies $S[a]$, where $a$ is non-separating and $g(S)\geq 2$, and compression bodies of the form $S[a_1,a_2]$, where $a_1,a_2$ are disjoint, separating curves on $S$.
\end {enumerate}
\end{Cor}
\begin {proof}
By Proposition \ref{prop:height}, height one compression bodies are minimal compression bodies, so (1) follows from Corollary \ref{cor:minimal}.

A non-separating small compression body $S[a]$, where $g(S)\geq 2$, has height $2$, since its interior boundary is connected with genus one less than that of the exterior boundary.

A pair of disjoint separating curves $a,b\subset S$ separates $S$ into three subsurfaces $S_1,S_2,S_3$, where $\sum_i g(S_i)=g(S)$. All of these have positive genus, so $S[a,b]$ has three interior boundary components, with these same genera. So, $\mathfrak{h}(S[a,b])=2$.

Finally, if a compression body $C$ has height two, then by Corollary \ref{lem:curvetosmc} there must be a pair of disjoint curves $a_1,a_2$ satisfying $(*)$ with $S[a_1,a_2]=C$. It follows from $(*)$ that $g(S)\geq 2$, and that $a_1$ is separating. If $a_2$ is non-separating, then $(*)$ implies that the component of $S\setminus a_1$ containing $a_2$ is a punctured torus, in which case $S=S[a_2]$.
\end {proof}

\subsection{Separating and non-separating compression bodies}

In this section, it is shown that for every compression body there exists a compressing system consisting of entirely non-separating curves or entirely separating curves.  Furthermore, this dichotomy is determined by the presence or lack, respectively, of a non-separating meridian.

\begin {figure}
\centering
\includegraphics{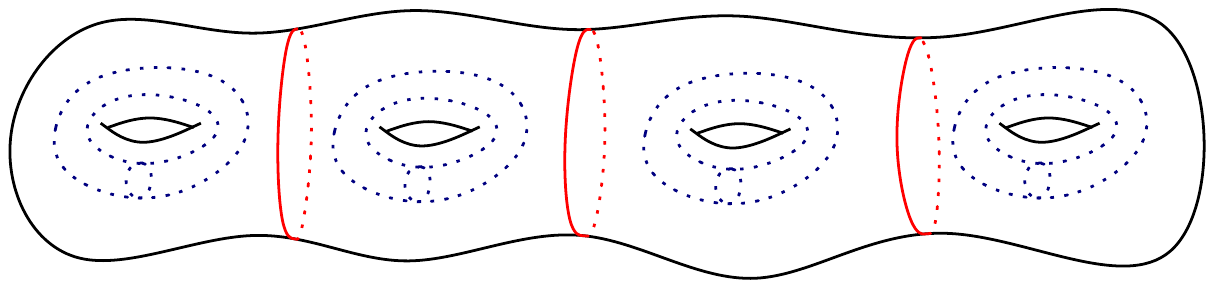}
\caption {An $S$-compression body that has $g(S)$ interior boundary components, all of which are tori. A compressing system is drawn in red, and the interior boundary is blue. All such compression bodies are isomorphic, although not all compressing systems look like the above.}
\label {maxtori}
\end {figure}

\begin {Prop}\label {prop:sep}
For an $S$-compression body $C$, the following seven conditions are equivalent:
\begin {enumerate}
\item $C =S[a_1,\ldots,a_k]$, where each $a_i$ is separating,
\item $H_1(S)\longrightarrow H_1 (C)$ is injective (and hence an isomorphism),
\item every meridian of $C$ is separating,
\item solid tori are never used in any sequence of minimal compressions for $C$,
\item $C$ has a sequence of minimal compressions in which solid tori are never used,
\item the number of interior boundary components of $C$ is $\mathfrak{h}(C)+1$,
\item the genera of the interior boundary components of $C$ sum to $g(S)$,
\item $C$ is contained in a compression body that has $g(S)$ interior boundary components, all of which are tori, as pictured in Figure \ref{maxtori}.
\end {enumerate}

\end{Prop}
\begin {proof}
For $(1)\implies (2)$  the kernel of $\pi_1 S \longrightarrow\pi_1 C$ is normally generated by the curves $a_1,\ldots,a_k$, which all lie in the commutator subgroup $[\pi_1 S ,\pi_1 S ]$. So, the entire kernel lies in $[\pi_1 S ,\pi_1 S ]$, implying that the induced map $H_1(S)\longrightarrow H_1 (C)$ is injective. 

For $(2)\implies (3)$, note that any non-separating curve is nontrivial in $H_1(S)$, so by $(2)$ cannot be trivial in $H_1(C)$.

$(3) \implies (1)$ are trivial.

$(3)\implies (4)$, since gluing a solid torus to an interior boundary component $F \subset\partial_- C$ compresses a non-separating curve on $F$, which is then homotopic to a non-separating meridian on $S $. 

$(4)\implies (5)$ is trivial.

$(5)\implies (6)$, since gluing a separating small compression body onto an interior boundary component $F$ removes $F$, but contributes two new interior boundary components.

$(6)\implies (7)$, by the definition of height.

$(7)\implies (8)$, since to each interior boundary component $F $ of $C $, we can glue an $F$-compression body with $g(F)$ interior boundary components, all of which are tori.

For $(8) \implies (2)$, note that the compression body in Figure \ref{maxtori} has a compressing system (pictured) consisting of only separating curves, so all its meridians are separating by the fact that $(1)\implies (2)$. The same is then true for any sub-compression body.
\end {proof}

The next result is a corollary of Lemma \ref{compressdiscs} and Corollary \ref{onlyone}.

\begin{Cor}
\label{cor:nonsep}
An $S$-compression body $C$ can be written as $S[a_1, \ldots, a_m]$ with each $a_i$ non-separating if and only if $C$ contains a non-separating meridian.
\end{Cor}

\begin{figure}
\centering
\includegraphics{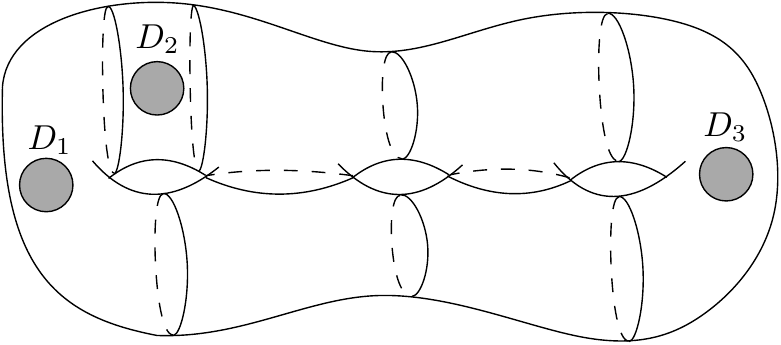}
\caption{A handlebody with spots $D_1, D_2, D_3$ corresponding to the attachment of 1-handles.  The non-separating meridians shown on the surface give a pants decomposition of the punctured surface obtained by deleting the spots from the boundary of the handlebody.}
\label{fig:non-sep-cb}
\end{figure}

\begin{proof}

The forward implication is obvious, so assume that $C$ has a non-separating meridian.
As  described in Corollary \ref{onlyone}, we can construct $C$ by attaching, for each component of $\partial_-C$, an interval bundle with a single 1-handle to a handlebody $H$.  
Since $C$ has a non-separating meridian, $H$ must have positive genus, for otherwise $H$ would admit a compressing system with only separating meridians, violating Proposition \ref {prop:sep} (3).

Each 1-handle intersects $\partial H$ in a disk; let $\{D_1, \ldots, D_n\}$ be the collection of theses disks. Choose any pants decomposition $\{a_1, \ldots, a_m\}$ for the punctured surface $\partial H \smallsetminus \left(\bigcup_i D_i\right)$ in which all the $a_i$ are non-separating (as in Figure \ref{fig:non-sep-cb}); we claim that $C = S[a_1, \ldots, a_k]$.  To see this, observe that $\{a_1, \ldots, a_m, \partial D_1, \ldots, \partial D_n\}$ is a maximal set of disjoint meridians for $C$, and for each $1\leq i\leq n$ there exists $1\leq j\neq k \leq m$ such that $\partial D_i$ bounds a pair of pants with some $a_j$ and $a_k$.  It follows that
$$C=S[\partial D_1, \ldots, \partial D_n, a_1, \ldots, a_k] = S[a_1, \ldots, a_k]$$
where the first equality is due to Lemma \ref{compressdiscs}.
\end{proof}

\section{Small compression bodies are invariant}
\label{small}

The main goal of this section is to show that an automorphism of $\cb(S)$ sends small compression bodies to small compression bodies with the same type:

\begin{Prop}
\label{prop:small-invariant}
Suppose that $f$ is an automorphism of the compression body graph $\cb(S)$. If $a$ is a simple closed curve on $S$, then $f(S[a])=S[b]$ for some simple closed curve $b$. Moreover, $a$ is separating if and only if $b$ is separating.
\end{Prop}

Recall from \S \ref{sec:height} that the \emph {height}, written $\mathfrak{h}(C)$, of an $S$-compression body $C$ is the length $k$ of any sequence of minimal compressions $$S \times [0,1] =C_0 \subset \cdots C_k =C.$$ Height one compression bodies are solid tori and small compression bodies $S[a]$, where $a$ is separating. Height two compression bodies are of the form $S[a]$, where $a$ is non-separating, or of the form $S[a,b]$, where $a,b$ are separating. See Corollary \ref{cor:height}.

The bulk of the work in Proposition \ref {prop:small-invariant} is in the following result.

\begin {Prop}
\label{prop:preservation}
Every automorphism of the compression body graph preserves height.
\end {Prop}

The proof of Proposition \ref{prop:preservation} will occupy most of this section; the idea is that  the height of $C$ is encoded in the chromatic numbers of certain subsets of the link of $C$ in $\cb(S)$. We describe the structure of links in \S \ref{sec:links}, and finish the proof of height preservation in \S \ref{sec:preservation}.

The invariance of small compression bodies almost follows from Proposition \ref{prop:preservation}, but one must also show that non-separating small compression bodies are not sent by $f$ to compression bodies $S[a,b]$, where $a,b$ are separating. For this, one can use:

\begin {Lem}\label {lem:sepheight}
Any height preserving automorphism $f : \cb(S) \longrightarrow\cb(S)$ preserves the set of compression bodies $C $ with only separating meridians.
\end{Lem}

Such $C$ admit a number of different characterizations, see Proposition \ref{prop:sep}. In particular, by $(1)\implies (3)$, a height two compression body $S[a,b]$, with $a,b$ separating, has only separating meridians. A non-separating small compression body clearly does not.

\begin {proof}

Note that $f$ preserves containment, since one can distinguish between the edge relations $C \subset D$ and $D \subset C$ using height.  It must then preserve the set of chains $C_1 \subset C_2 \subset\cdots$ of compression bodies, and therefore the set of maximal such chains, i.e.\ sequences of minimal compressions. Also $f$ sends handlebodies to handlebodies, since these are exactly the compression bodies with height $2g(S)-1$.  

It follows that $f$ preserves the set of compression bodies $C$ that have only torus interior boundary components, as these $C$ can be characterized by the fact that after choosing a handlebody $H \supset C$, there are only finitely many sequences of minimal compressions for $H$ that pass through $C$. Here, solid tori can be attached to the interior boundary components of $C $ in any order, but the attachment maps are prescribed by $H$. If $C$ has a higher genus interior boundary component $F$, though, there are infinitely many intermediate compressions to choose from when attaching a handlebody to $F$.

In particular, $f$ must preserve the height $g(S)-1$ compression bodies $C$ that have only torus interior boundary components, as shown in Figure \ref{maxtori}. Here, the number of boundary components is $g(S)$, since $g$ additional compressions are required to reach a handlebody, which has height $2g(S)-1$.  By Proposition \ref{prop:sep}, the compression bodies that are contained in such $C$ are exactly those that have only separating meridians.\end {proof}

The rest of the section is devoted to the proof of Proposition \ref {prop:preservation}, which states that automorphisms of the compression body graph preserve height.

\subsection{Links in $\cb(S)$}
\label {sec:links}
Given a graph $\Gamma$ we denote the edge relation in $\Gamma$ by $\sim_\Gamma$.  We will simply use $\sim$ when the graph is clear from context.
The \textit{join} of two graphs $\Gamma$ and  $\Gamma'$, written $\Gamma+\Gamma'$, is the graph with vertex set $\Gamma \sqcup \Gamma'$, and where vertices $v, w$ are adjacent if either:
\begin{itemize}
\item[(i)]
$v,w\in \Gamma$ and $v\sim_\Gamma w$, or
\item[(ii)]
$v,w\in \Gamma'$ and $v\sim_{\Gamma'} w$, or
\item[(iii)]
$v\in \Gamma$ and  $w\in \Gamma'$.
\end{itemize}
In particular, note that if two vertices in a graph join are not connected by an edge, they must lie in the same factor. Here is a useful consequence:

\begin {Fact}[Uniqueness of join]\label {fact:uniqueness}
Suppose $\Gamma,\Gamma'$ are anti-connected graphs, and $\Delta,\Delta'$ are arbitrary. If $\Gamma + \Gamma' =\Delta + \Delta'$, then up to exchanging factors, $\Gamma=\Delta$ and $\Gamma'=\Delta'$.
\end {Fact}

Here, a graph is \emph {anti-connected} if any two vertices can be connected by an anti-path, i.e.\ a sequence of vertices $(v_i)$ where $v_i \not\sim v_{i+1}$ for all $i $.  To prove the fact, just assume there is some $v\in \Delta \cap \Gamma$ and note that any anti-path starting at $v$ stays in $\Delta$.

As mentioned above, the link of a compression body $C \in\cb(S)$ decomposes as a join
$$\mathrm{Link}(C)=\mathrm{Link}^{-}(C)+\mathrm{Link}^{+}(C),$$ where
$$\mathrm{Link}^{+}(C) = \{D\in \cb(S) \colon C\subset D\}$$
and
$$\mathrm{Link}^{-}(C) = \{D\in \cb(S) \colon D\subset C\}$$
are the \emph{uplink} and \emph{downlink} of $C$, respectively.

\begin{Lem}
\label{lem:uniqueness}
If $C\in \cb(S)$, the graphs $\mathrm{Link}^{+}(C)$ and $\mathrm{Link}^{-}(C)$ are anti-connected. 
\end{Lem}

So by Fact \ref {fact:uniqueness}, the only way to write $\mathrm{Link}(C)$ as a join is using  the uplink and downlink.

\begin{proof}
We'll say two vertices $v,w$ are \emph {anti-adjacent} if $v \not \sim w$ in $\cb(S)$.

We first deal with $\mathrm{Link}^{+}(C)$.  Pick a handlebody $H \supset C$ that lies in $\Gamma$. By Corollary \ref{sum}, $D$ is obtained from $C $ by attaching handlebodies (with smaller genus) to the components of $\partial_- C $. In particular, there is a non-separating simple closed curve $a $ on some component $F \subset \partial_-C$ that bounds a disk in $H$. 

Choose a simple closed curve $b$ on $F$ with $i(a,b)=1$, and let $C[b]$ be the compression body obtained from $C$ by compressing $b$. Then $C[b]$ and $H$ are anti-adjacent, since two meridians in a handlebody cannot intersect once.  Moreover, every compression body containing $C[b]$ is anti-adjacent to $H$ as well.

Any two compression bodies with the same height are anti-adjacent. The compression body $C[b]$ and its uplink represent all heights in $\mathrm{Link}^{+}(C)$, except $\mathfrak{h} (C)+1$ when $g(F)\geq 2$. In this last case, though, there is a separating simple closed curve $c$ on $F$ that is not a band sum with $b$, and then $C[c] $ is a height $\mathfrak{h} (C)+1$ compression body that is anti-adjacent to $ C[b]$. Therefore, the uplink $\mathrm{Link}^{+}(C)$ is anti-connected.

The argument for $\mathrm{Link}^{-}(C)$ is similar. Start with a compression body $D\subset C$ with height $\mathfrak{h} (C)-1$, and pick some meridian $a$ of $C$ that is not a meridian in $D$. Then every sub-compression body of $C$ that contains $a$ is anti-adjacent to $D$. These fill out all heights in $\mathrm{Link}^{-}(C)$, except height one if $a$ is not separating. In this latter case, $\mathfrak{h} (C)-1\geq 2$, so $C$ must have a separating meridian $b$ that is not a band sum with $a$. Then $S[b]$ is a height one compression body that is anti-adjacent to $S[a]$, and the downlink is anti-connected.
\end{proof}

A \textit{clique} in a graph $\Gamma$ is a complete subgraph and the \textit{clique number}, written $\omega(\Gamma)$, is the number of vertices in a maximal clique.  A \textit{proper coloring} of a graph is a labeling of the vertices such that vertices connected by an edge are assigned different labels.  The minimal number of colors required to give a proper coloring of $\Gamma$ is the \textit{chromatic number}, written  $\chi(\Gamma)$.  It is clear that 
\begin{equation}
\label{eq:clique}
\omega(\Gamma) \leq \chi(\Gamma).
\end{equation}

\begin{Lem}
\label{lem:chromatic}
If $C\in \cb(S)$, the clique and chromatic numbers satisfy:
$$\omega(\mathrm{Link}^{-}(C)) = \chi(\mathrm{Link}^{-}(C)) = \mathfrak{h}(C) - 1$$
and
$$\omega(\mathrm{Link}^{+}(C)) = \chi(\mathrm{Link}^{+}(C)) = 2g-1 - \mathfrak{h}(C).$$
\end{Lem}

\begin{proof}
Labelling a vertex of $\mathrm{Link}^{-}(C)$ by its height is a proper coloring, so $\chi(\mathrm{Link}^{-}(C)) \leq \mathfrak{h}(C)-1$. The induced subgraph on any maximal chain of sub-compression bodies $ C_1 \subset \cdots \subset C_{\mathfrak{h}(C)-1}=C$ is a complete graph, so $\omega(\mathrm{Link}^{-}(C)) \geq \mathfrak{h}(C) - 1$.  
Therefore,
$$\omega(\mathrm{Link}^{-}(C)) = \chi(\mathrm{Link}^{-}(C)) = \mathfrak{h}(C) - 1$$
by \eqref{eq:clique}. The case for uplinks is similar.
\end{proof}

\subsection{Automorphisms preserve height: the proof of Proposition \ref {prop:preservation}.}
\label{sec:preservation}

Fix an automorphism $f\in \Aut(\cb(S))$.  If $C\in \cb(S)$, we can write
$$\Link(f(C)) = \mathrm{Link}^{+}(f(C)) + \mathrm{Link}^{-}(f(C)) = f(\mathrm{Link}^{+}(C)) + f(\mathrm{Link}^{-}(C)).$$
The uniqueness of the join (Lemma \ref{lem:uniqueness}) that implies that either:
\begin{itemize}
\item[(i)]
$f(\mathrm{Link}^{+}(C)) = \mathrm{Link}^{+}(f(C)) \ \text{ and } \ f(\mathrm{Link}^{-}(C)) = \mathrm{Link}^{-}(f(C))$, or
\item[(ii)]
$f(\mathrm{Link}^{+}(C)) = \mathrm{Link}^{-}(f(C))  \ \text{ and } \  f(\mathrm{Link}^{-}(C)) = \mathrm{Link}^{+}(f(C)).$ 
\end{itemize}

\begin {Lem}\label {iii}
Either $(i)$ holds for all $C\in \cb(S)$, or $(ii)$ holds for all $C\in \cb(S)$.
\end {Lem}
\begin {proof}
Since the graph $\cb(S)$ is connected, it suffices to show that when two compression bodies are adjacent, either $(i)$ holds for both or $(ii)$ holds for both. For convenience, let $$ {\mathrm{Link}^{\geq}(C)} = \{C\}\cup \mathrm{Link}^{+}(C) \ \text{ and } \ { \mathrm{Link}^{\leq}(C)} = \{C\}\cup \mathrm{Link}^{-}(C).$$ Then if $C \subset D$, the only inclusions that are present between the four sets
$$ {\mathrm{Link}^{\geq}(C)}, \  {\mathrm{Link}^{\leq}(C)}, \  {\mathrm{Link}^{\geq}(D)}, \  {\mathrm{Link}^{\leq}(D)}$$
are that $ {\mathrm{Link}^{\geq}(C)} \supset  {\mathrm{Link}^{\geq}(D)}$ and $ {\mathrm{Link}^{\leq}(C)} \subset  {\mathrm{Link}^{\leq}(D)}$.  Here, it is necessary to add in $C$ to the up and down links,for if $H$ is a handlebody, $\mathrm{Link}^{+}(H)$ is empty, and is included in all four sets above. Similarly, $\mathrm{Link}^{-}(H)$ is empty when $C$ is a separating small compression body.

In conclusion, it cannot be the case that up and down links are switched for $C$, but preserved for $D$ (or vice versa), since then when considering $f(C)$ and $f(D)$, one would see some $\mathrm{Link}^{\geq}$ included in a $\mathrm{Link}^{\leq}$ (or vice versa).
\end {proof}

By Lemma \ref{lem:chromatic}, the height of a compression body can be calculated from the chromatic number of its uplink, or of its downlink. So, in light of Lemma \ref{iii}, $f$ is either height preserving (in which case we are done) or `height reversing', that is
\begin {equation}\mathfrak{h}(f(C)) = 2g-1-\mathfrak{h}(C), \text{ for all } C\in \cb(S).\label {hr}\end{equation}

Assuming \eqref{hr}, we break into cases.  When $g(S)\geq 3$, Corollary \ref{cor:height} implies that there are two types of height $2$ compression bodies: non-separating small compression bodies, and compression bodies of the form $S[a,b]$, where $a,b$ are separating and disjoint. By Lemma \ref{lem:sepheight}, any height preserving automorphism of $\cb(S)$ preserves these two types. On the other hand, every compression body of height $2g-2$ has a single torus interior boundary component, so by Corollary \ref {onlyone}, they are all homeomorphic. Therefore, $\Mod(S)$ acts transitively on height $2g-2$ compression bodies. This action is conjugated by $f$ to a transitive, height preserving action on height two compression bodies, a contradiction.

When $S$ has genus two, the argument above fails since a genus two surface does not admit a pair of disjoint separating curves.  Here, there are three possible heights:
\begin {description}[labelindent=.5cm]
\item [$(\mathfrak h = 1)$]  separating small compression bodies,
\item [$(\mathfrak h = 2)$]  non-separating small compression bodies,
\item [$(\mathfrak h = 3)$]  handlebodies.
\end {description}
As $f$ is height reversing, the set of non-separating small compression bodies is left invariant, while separating small compression bodies and handlebodies are exchanged.  Note that since $f$ reverses height, it reverses the order of containment: $C \subset D \implies f(C)\supset f(D)$. 

\begin{figure}
\centering
\includegraphics{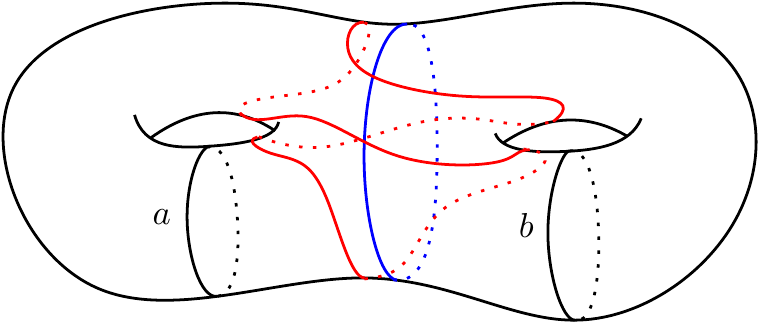}
\caption{The two curves above are both meridians in only a single handlebody, while there are infinitely many separating curves that are disjoint from both. }
\label{fig:genustwo}
\end{figure}

Let $a,b$ be the curves indicated in Figure \ref{fig:genustwo}, and let $$f(S[a])=S[a'], \ \ f(S[b])=S[b'].$$
As $S[a]$ and $S[b]$ are contained in a common handlebody, $S[a']$ and $S[b']$ both contain the same separating small compression body, i.e. there are punctured tori $T_a$ and $T_b$ containing $a',b'$ with $\partial T_a=\partial T_b$. If $T_a=T_b$, then $a',b'$ have nonzero algebraic intersection number, so $S[a']$ and $S[b']$ cannot be contained in the same handlebody, contradicting that $a,b$ both contain a common separating small compression body. Therefore, we must have $T_a\neq T_b$, in which case the configuration of $a',b'$ is the same, up to homeomorphism, as that of $a,b$. But this is a contradiction, since then $S[a'],S[b']$ contain infinitely many common separating small compression bodies, while $S[a],S[b]$ are contained in only a single handlebody.

\section{The torus complex}
\label{torus}

The \textit{torus complex} associated to a surface $S$, written $\tc(S)$, is the complex whose vertices are isotopy classes of simple non-separating curves and where the collection of vertices $\{a_0, \ldots, a_k\}$ spans a $k$-dimensional simplex if there exists an a punctured torus $T \subset S$ with $a_i\subset T$ for every $1\leq i \leq k$.  Note this is an infinite dimensional complex.  

As described in the introduction, the automorphism group of the compression body graph can be determined using the following theorem.

\vspace{11pt}
\textbf{Theorem \ref{thm:torus-aut}}
\textit{
The natural map $\Mod^\pm(S) \to \Aut(\tc(S))$ is a surjection.
}
\vspace{11pt}

As with automorphisms of the curve complex and of $\cb(S)$, the map is an isomorphism unless $S$ has genus two, in which case the kernel is generated by the hyperelliptic involution. 
We will prove surjectivity by showing that an automorphism of $\tc(S)$ induces an automorphism of the \emph {Schmutz graph}. Here, the Schmutz graph $\cN(S)$ has the same vertex set as $\tc(S)$, but edges connect pairs of vertices that intersect once. Schmutz~\cite{schmutzmapping} proved that every automorphism of $\cN(S)$ is induced by a mapping class.

We will need the following definition.

\begin{Def}
A \emph {triangle} in $\tc(S)$ is a triple of vertices $a,b,c$, where each pair of vertices is connected by an edge. A triangle is \textit{empty} if it does not bound a 2-simplex.
\end{Def}

\begin{figure}
\centering
\includegraphics{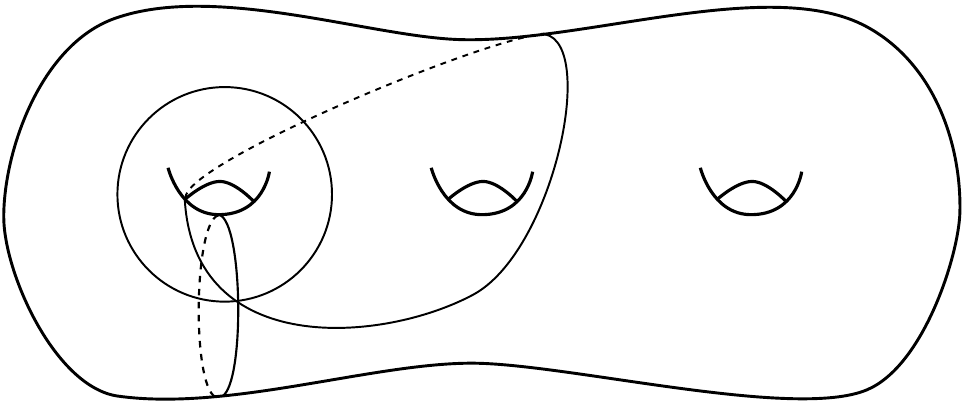}
\caption{An example of three curves forming an empty triangle in $\tc(S)$.}
\label{fig:empty-tri}
\end{figure}

Suppose $a,b$ are adjacent vertices in $\tc(S)$. We claim that if $i(a,b)=1$, there exists a punctured torus $T \subset S$ such that there are infinitely many empty triangles $a,b,c$ with $a,c$ contained in $T$, while if $i(a,b)\neq 1$, there are at most three empty triangles for every $T$. As this characterization concerns just the simplicial structure of $\tc(S)$, it proves that intersection number one is preserved by simplicial automorphisms.  So by Schmutz \cite{schmutzmapping},  any automorphism of $\tc(S)$ is induced by a mapping class.

Assuming $i(a,b)=1$, the empty triangles are easy to construct. Choose any simple closed curve $d$ in $S$ that intersects $a$ once, is disjoint from $b$, but is not homotopic into a regular neighborhood $N$ of $a \cup b.$ 
If $T$ is a regular neighborhood of $a \cup d$, then twisting $a$ around $d$ gives infinitely many simple closed curves $c$ in $T$ that lie in punctured tori with both $a$ and $b$. Any such $c$ determines an empty triangle with vertices $a,b,c$.

For the other direction, we will need the following result.

\begin{Prop}
\label{prop:intersection}
At most one edge of an empty triangle in $\tc(S)$ can connect simple closed curves on $S$ that intersect more than once.
\end{Prop}

Deferring the proof for a moment, let's finish the characterization of pairs of vertices that intersect once.  Suppose $i(a,b)>1$.  If $a,b,c$ is an empty triangle, then by Proposition~\ref{prop:intersection}, we must have $i(a,c)=i(b,c)=1.$  But if $T $ is a punctured torus containing $a$, there are at most $3$ simple closed curves $c \subset T$ that intersect both $a,b$ once. 

This is easiest to see in coordinates.  After picking a basis for the homology of $T$, proper arcs and curves in $T$ can be labeled by extended rational numbers $\frac pq\in \mathbb Q \cup \infty$. Here, the $\frac pq$-arc is the unique arc that is disjoint from the $\frac pq$-curve.  See Figure \ref{fig:combined}\subref{fig:torus}. 
The intersection number of the $\frac pq$-curve and the $\frac mn $-curve is $|pn-mn|$, and the same formula holds for intersections of arcs and curves (although for intersections of two arcs it is off by one).
So, assume $a$ is the $\frac 10$-curve, and that some component of $b \cap T_3$ is the $\frac pq$-arc. As $i(a,b)>0$, we can assume that this arc intersects $a$, i.e.\ that $q\neq 0$.  Then if the $\frac mn$-curve in $T$ intersects $a$ and $b$ once, it intersects the $\frac pq $-arc at most once, so
$$ |1\cdot n-0\cdot m | = 1,  \ \ |p\cdot n-q\cdot m| \leq 1.$$
These conditions are only satisfied when $q=\pm 1$, in which case $\frac mn$ must be $\frac pq$ or $\frac {p\pm 1}q$.

\subsection {The proof of Proposition \ref {prop:intersection}}

We require the following lemma.

\begin{figure}[t!]
\centering
\includegraphics{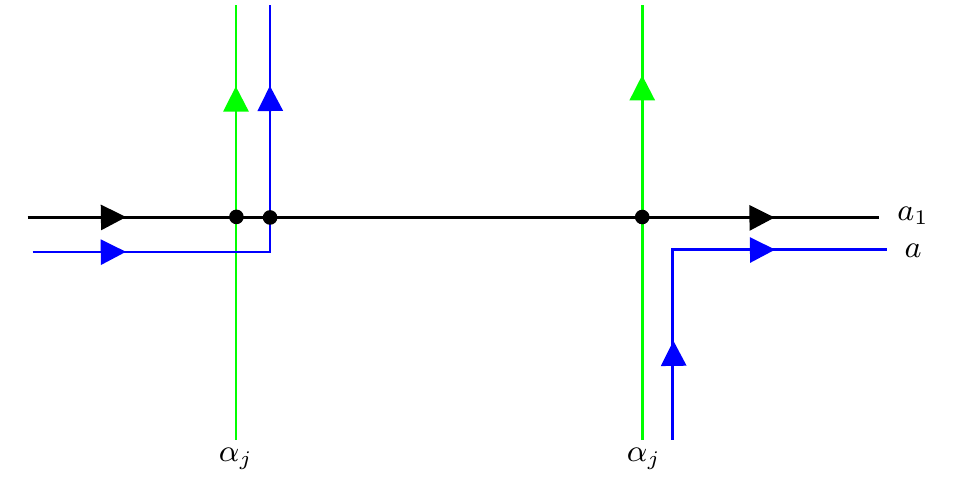}
\caption{Shown here are two consecutive intersections between $a_1$ (black) and $\al_j$ (green) and the  curve $a$ (blue) resulting from the following surgery:  Begin at the intersection on the left, follow $\al_j$ to the second intersection, and then following $a_1$ back to the first intersection.  Notice that $i(a_1,a) = 1$.}
\label{fig:surgery}
\end{figure}

\begin{Lem}
\label{lem:arcs}
If $a_1, a_2\in \tc(S)$ are contained in a punctured torus $T \subset S$, then for any punctured torus $T'\neq T$ in $S$ containing $a_1$, 
$$a_2 \cap T' = \al_1 \sqcup \cdots \sqcup \al_n$$
where $\al_j$ is a simple proper arc of $T'$ satisfying $i(a_1, \al_j) \leq 1$ for each $1\leq j\leq n$.
\end{Lem}

\begin{proof}
Suppose $i(a_1, \al_j) \geq 2$ for some $1\leq j \leq n$, then we may perform the surgery shown and described in Figure \ref{fig:surgery}.  The resulting curve $a$ satisfies $i(a_1,a) = 1$ implying it is in minimal position with respect to $a_1$ allowing us to conclude that $a\neq a_1$.  As this surgery occurred in $T'$ it is clear that $a \subset T'$.  Furthermore, as $a$ is obtained from a surgery on $a_1$ and $a_2$, $a\subset T$.  Now $T$ and $T'$ share two simple closed curves implying $T= T'$, which is a contradiction.
\end{proof}

\begin{figure}[t]%
	\centering
	\subfloat[\label{fig:torus}]{{
		\includegraphics{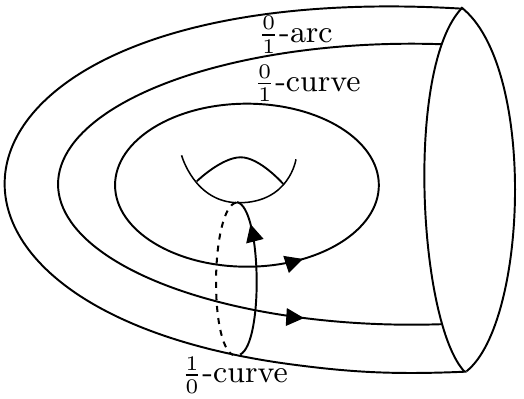}}}%
	\qquad\quad
	\subfloat[\label{fig:orientation}]{{
		\includegraphics{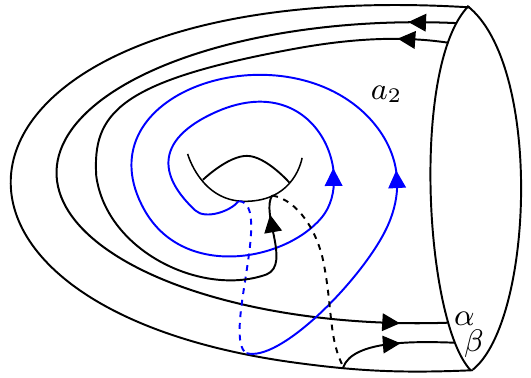}}}%
		
	\subfloat[\label{fig:components}]{{
		\includegraphics{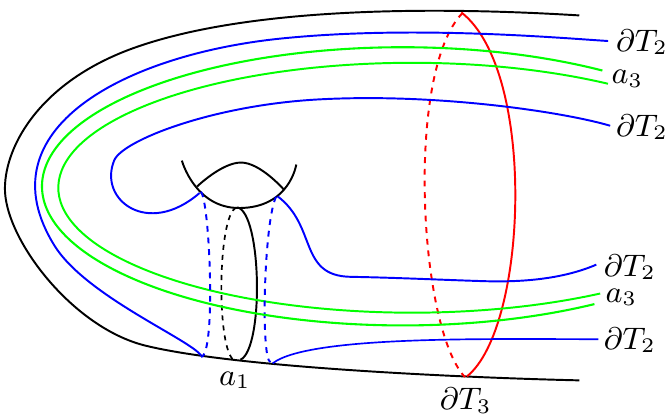}}}%
	\caption{(a) A reference torus. 
	(b) Drawn here is $a_2, \al,$ and $\be$ with $a_2$ given an arbitrary orientation.  As $a_1$ and $a_3$ live in a torus, the orientations of the intersections of $\al$ and $\be$ with the $\frac10$-curve must agree.
	(c) The region containing $a_1$ that is bounded by the blue and red curves is the intersection $T_2 \cap T_3$.}
	
	\label{fig:combined}
\end{figure}

\begin{proof}[Proof of Proposition \ref{prop:intersection}]
Label the vertices of the empty triangle as $a_1, a_2, a_3\in \tc(S)$. We will show that if $i(a_1,a_2)>1$, then $i(a_1,a_3) = 1$.  

Let $T_i\subset S$ be the punctured torus containing $a_j$ and $a_k$ for $i\neq j\neq k \in \{1,2,3\}$. Then in light of Lemma \ref{lem:arcs}, it is enough to show that $a_3\cap T_3$ has a single component.

As above, we will work in coordinates, labeleling arcs and curves in $T_3$  by  extended rational numbers $\frac pq\in \mathbb Q \cup \infty$. 
We will assume $a_1$ is the $\frac10$-curve in $T_3$.  Since $i(a_1,a_3) > 0$, Lemma \ref{lem:arcs} guarantees that there exists a component $\al$ of $a_3\cap T_3$ that intersects $a_1$ exactly once; without loss of generality, we may assume $\al$ is the $\frac01$-arc.
As $i\left(a_1, a_2\right) \neq 1$, we know $a_2$ is not the $\frac01$-curve; in particular, $i\left(a_2, \al\right) = 1$  by Lemma \ref{lem:arcs} as $a_2, a_3 \subset T_1$.  This forces $a_2$ to be the $\frac1m$-curve for some $m\in \bz$ with $|m|\geq 2$. 

We now want to rule out the existence of other components $\be$ of $a_3\cap T_3$. Any such $\be$ intersects both $a_1$ and $a_2$ once.  For if $i(a_1, \be) = 0$, then $\be$ is the $\frac10$-arc, implying $i(a_2, \be) > 1$, contradicting Lemma \ref{lem:arcs}.  So, we can conclude $i(a_1,\be) = 1$, again by Lemma \ref{lem:arcs}.  Similarly, we have $i(a_2, \be) = 1$.

Suppose first that $\be$ is not isotopic to $\al$.  Then as $i(a_1, \be) = 1$, we know $\be$ is the $\frac n1$-arc for some $n\neq 0 \in \bz$. We now have two integers $m,n$ satisfying 
$$|m\cdot n -1| = 1,$$
as $i(a_2,\al) = 1$.  
Since $n\neq 0$, 
$$(m,n) \in \{(2,1), (-2,-1)\}$$
since $|m|\geq 2$.

Consider the case $(m,n) = (2,1)$, where $a_2$ is the $\frac12$-curve and $\be$ is the $\frac11$-arc, see  Figure~\ref{fig:combined}\subref{fig:orientation}. Choose orientations for $a_1, a_2$ and $a_3$.  As $a_1$ and $a_3$ live in a punctured torus, we know that the orientations of the intersections of $\al$ and $\be$ with $a_1$ must agree (as in Figure \ref{fig:combined}\subref{fig:orientation}).  However, this forces the orientations of the intersections of $\al$ and $\be$ with $a_2$ to disagree, which contradicts $a_2$ and $a_3$ being contained in the punctured torus $T_1$.  A similar argument implies that $(m,n) \neq (-2,-1)$.

We have now shown that all the components of $a_3\cap T_3$ are isotopic to the $\frac 01$-arc in $T_3$.  When $T_2$ and $T_3$ are put in minimal position, the intersection $ T_2\cap T_3$ must then be exactly as shown in Figure \ref{fig:combined}\subref{fig:components}, since any component of $T_2 \cap T_3$ must contains some component of $a_3 \cap T_3$.  Then $R=T_2 \cap (S\smallsetminus T_3)$ is a rectangle, and $a_3 \cap R$ is a collection of parallel arcs. Since $S$ is orientable, the parallel arcs in $a_3 \cap T_3$ and $a_3 \cap R$ glue to a collection of parallel loops. But $a_3$ is supposed to be a simple closed curve, so $a_3 \cap T_3$ must have a single component intersecting $a_1$ once.
\end{proof}

\section{Automorphisms as mapping classes}

\label{sec:proof}
In this section we complete the proof of Theorem \ref{main}, that is to say that the natural homomorphism $\Mod(S) \to \Aut(\cb(S))$ is a surjection.

Let $f$ be an automorphism of $\cb(S)$.  By Proposition \ref{prop:small-invariant}, $f$ permutes the small compression bodies $S[a]\in \cb(S)$, so the formula $f(S[a])=S[f_*(a)]$ defines a map
$$f_* : \{\text {simple closed curves on } S\} \longrightarrow \{\text {simple closed curves on } S\}.$$
Moreover, Proposition \ref{prop:small-invariant} says that $a$ is non-separating if and only if $f_*(a)$ is.
Since a non-separating curve $a \subset S$ is contained in a punctured torus $T\subset S$ if and only if $S[\partial T] \subset S[a] $, $f_*$ preserves when a collection of non-separating curves is contained in a punctured torus. So, $f_*$ extends to an automorphism of the torus complex $\tc(S)$.  As every automorphisms of $\tc(S)$ agrees with a mapping class (Theorem \ref{thm:torus-aut}), the action of $f$ on the set of non-separating small compression bodies agrees with a mapping class.

By post-composing $f$ with a mapping class, we obtain an automorphism of $\cb(S)$ fixing all non-separating small compression bodies.  We claim:

\begin {Prop}\label {fixnonsep}
The only automorphism of $\cb(S)$ that fixes all non-separating small compression bodies is the identity.
\end {Prop}

This will imply that our $f$ above agrees with a mapping class, and will finish the proof of Theorem \ref{main}.
To prove Proposition \ref{fixnonsep}, we must set up some terminology.

Let $\Sigma$ be an orientable finite-type surface, possibly with boundary.  The \emph{curve graph} $\cc(\Sigma)$ is the graph whose vertices are isotopy classes of essential non-peripheral simple closed curves in $\Sigma$ and where edges connect pairs of simple closed curves that intersect minimally. (Note that unless $S$ is a torus, a punctured torus or a $4$-holed sphere, there are pairs of disjoint curves on $S$, so then `intersect minimally' means disjoint.)  The curve graph $\cc(\Sigma)$ has a natural metric $d_\cc$, determined by setting each edge to have length one.  Masur-Minksy \cite[Proposition 4.6]{Masurgeometry1} have shown that $\cc(\Sigma)$ has infinite diameter.

Let $F$ be an interior boundary component of a separating small compression body $S[a]$ and let $D_a$ be a properly embedded disk bounded by $a$. Let $\Sigma_F \subset S$ be the component of $S\smallsetminus a$ such that $\Sigma_F \cup D_a $ is isotopic to $F$ within $S[a]$. Note that as these surfaces are incompressible in $S[a]$, the isotopy gives a canonical identification $$\cc(\Sigma_F \cup D_a ) \overset{\cong}{\longrightarrow} \cc(F).$$

The \textit{interior boundary projection} from $S$ to $F$ is the multi-valued function $$\pi_F: \cc(S) \longrightarrow \cc(\Sigma_F \cup D_a ) \cong \cc(F)$$ defined as follows: let $b\in\cc(S)$ and assume $b$ is in minimal position with $a$.  Now,
\begin{itemize}
\item[(i)] 
if $b\in\cc(\Sigma_F)\subset \cc(S)$, let $\pi_F(b)=b$.
\item[(ii)]
if $b \cap \partial \Sigma_F$ is nonempty, then for each arc $\be$ of $b \cap \Sigma_F$, let $\be_1, \be_2\in \cc(\Sigma_F)$ be the two components of the boundary of a regular neighborhood of $\be \cup \partial \Sigma_F$.
Then,
$$\pi_F(b) = \bigcup_{\beta \subset b \cap \Sigma_F} \left \{\be_1, \be_2\right\}.$$
\item[(iii)]
otherwise, $\pi_F(b) = \emptyset$.
\end{itemize}

For the familiar, this is the same as the \emph {subsurface projection} from $S$ to $\Sigma_F$, except that at the end we cap off the boundary of $\Sigma_F$ with a disk.  The only fact we will need is:

\begin {Lem}\label {surjdisc}
If $S[a],S[b]$ are both contained in an $S$-compression body $ C$, there is component $F\subset\partial_-S[a]$ and an element $m \in \pi_F(b)$ that bounds a disk in $C \smallsetminus S[a]$.
\end {Lem}
\begin {proof}
The surgery of $b$ and $\partial \Sigma_F$ presented in Lemma \ref{innermost}(2) gives a meridian for $C$ that is contained in $\Sigma_F$, and isotoping this to $F$ gives an element of $\pi_F(b)$.
\end {proof}

\begin {proof}[Proof of Proposition \ref {fixnonsep}]
The proof will proceed in three stages: first, we show that our automorphism $f$ fixes all compression bodies that contain a non-separating meridian, then we show that this implies that $f$ fixes all small compression bodies (including the separating ones), and then we show that $f$ is the identity.

\begin{Lem}
\label{lem:nonsep-fixed}
An automorphism of $\cb(S)$ fixing every non-separating small compression body fixes every compression body containing a non-separating meridian.
\end{Lem}

\begin{proof}
If $C$ is an $S$-compression body containing a non-separating meridian, then $C = S[a_1, \ldots, a_m]$, where each $a_i$ is non-separating (Corollary \ref{cor:nonsep}).  If $f\in \Aut(\cb(S))$ fixes every non-separating small compression body, then $S[a_i] \subset f(C)$ for $1\leq i \leq m$.  In particular, $C\subset f(C)$, but $\mathfrak{h}(C) = \mathfrak{h}(f(C))$ (Proposition \ref{prop:preservation}) forcing $C = f(C)$. 
\end{proof}

\begin{Lem}
An automorphism of $\cb(S)$ that fixes every compression body that contains a non-separating meridian fixes every small compression body.
\end{Lem}

\begin {proof}
By Proposition \ref {prop:small-invariant}, the set of small compression bodies is invariant, so we must only show that separating small compression bodies are not nontrivially permuted.  So, we claim that if $S[a]$ and $S[b]$ are distinct separating small compression bodies, there is a compression body $C$ that has a non-separating meridian such that
$$S[a] \subset C,\ \text{ but } \ S[b] \not\subset C.$$
This will prove the lemma, since if $f(C)=C$ then we cannot have $f(S[a])=S[b]$.

 From the definition, it is easy to see that $\pi_F(b)$ has finite diameter in $\cc(F)$, so we can choose a non-separating curve $c\in \cc(F)$ satisfying
 \begin{equation}
\label{eq:diam}
d_\cc(c, \pi_F(b)) \geq 2.
\end{equation}
Let $C$ be obtained by gluing $F[c]$ to $S[a]$ (see Corollary \ref{sum}). If $S[b]\subset C$, Lemma~\ref{surjdisc} gives some $m\in \pi_F(b)$ that bounds a disc in $C\smallsetminus S[a]=F[c]$.  (Note: the $m$ given must lie on $F$, since the other component of $\partial_- S[a]$ is incompressible in $C$.) But by Lemma \ref{smalldisc}, every meridian in $F[c]$ is disjoint from $c$, so this would violate \eqref{eq:diam}.\end{proof}

Finally, recall that the \emph {disk set} of a compression body $C$, denoted $\cd(C)$, is the collection of all (isotopy classes of) meridians in $C$.  Then
\begin{itemize}
\item[(a)]
$S[a] \subset C$ if and only if $a \in \cd(C)$, and
\item[(b)]
$\cd(C) = \cd(D)$ if and only if $C$ and $D$ are isomorphic (Corollary \ref{disccont}).  
\end{itemize}
Therefore, an automorphism fixing every small compression body is the identity. 
\end {proof}

\textit{\textrm{
\bibliographystyle{amsplain}
\bibliography{biblio}
}}
\end{document}